\documentclass[twoside,11pt]{amsart}

\usepackage[cmtip,all]{xy}
\usepackage{amssymb, xspace, enumerate, times, color, thmtools, hyperref}

\usepackage{soul}
\usepackage[capitalize]{cleveref}
\usepackage{graphicx}

\input xypic
\input xy
\xyoption{all}
\def\sqr#1#2{{\vcenter{\hrule height.#2pt
        \hbox{\vrule width.#2pt height#1pt \kern#1pt
                \vrule width.#2pt}
        \hrule height.#2pt}}}

\theoremstyle{plain}
\newtheorem*{TheoremA}{Theorem A}
\newtheorem*{TheoremB}{Theorem B}
\newtheorem*{TheoremC}{Theorem C}

\newtheorem{Theorem}{Theorem}[section]
\newtheorem{Lemma}[Theorem]{Lemma}
\newtheorem{Corollary}[Theorem]{Corollary}
\newtheorem{Proposition}[Theorem]{Proposition}

\theoremstyle{definition}
\newtheorem{Notation}[Theorem]{Notation}
\newtheorem{Assumptions and Discussion}[Theorem]{Assumptions and Discussion}

\newtheorem{Remark}[Theorem]{Remark}
\newtheorem{Example}[Theorem]{Example}
\newtheorem{Definition}[Theorem]{Definition}
\newtheorem{Question}[Theorem]{Question}

\newtheorem{Construction}[Theorem]{Construction}

\def\m{{\mathfrak m}}

\def\q{{\mathfrak q}}
\def\p{{\mathfrak p}}

\newcommand{\symp}[1]{#1^{(\ell)}}

\newcommand{\skel}[1]{^{[#1]}}
\newcommand{\dual}{^*}

\newcommand{\supp}[1]{{\rm supp}(#1)}
\newcommand{\msupp}[1]{{\rm m\!\!-\!\!supp}(#1)}

\newcommand{\term}[1]{\textbf{{#1}}}
\newcommand{\sfp}{SF}

\def\ZZ{{\mathbb Z}}
\def\NN{{\mathbb N}}

\def\kk{{\mathbb K}}

\newcommand{\mc}[1]{\mathcal{#1}}

\newcommand{\C}{\mc{C}}
\newcommand{\B}{\mathcal{B}}
\newcommand{\M}{\mathcal{M}}

\newcommand{\F}{\mathcal{F}}
\newcommand{\FF}{\mathbb{F}}

\def\Llra{\Longleftrightarrow}

\def\Lra{\Longrightarrow}
\def\lra{\longrightarrow}

\newcommand{\be}{\begin{equation*}}
\newcommand{\ee}{\end{equation*}}
\newcommand{\bee}{\begin{equation}}
\newcommand{\eee}{\end{equation}}

\def\h{{\rm ht}}

\def\Ass{{\rm Ass}}

\def\reg{{\rm reg}}

\def\LCM{{\rm LCM}}
\def\GCD{{\rm GCD}}

\def\sub{\subseteq}
\def\sup{\supseteq}

\def\vs{V} 
\def\con{/} 
\def\div{*} 
\def\ineq{\preceq} 
\def\stineq{\prec} 

\def\facet{\mathcal{F}} 

\newcommand\cm[1]{\mathcal{M}{(#1)}} 
\newcommand\ccm[1]{\cm{#1}^{+}} 
\newcommand\zcm[1]{\cm{#1}^{0}} 
\newcommand\cmname{{focal }} 
\newcommand\col[1]{C_{#1}} 
\newcommand\shift[1]{(-#1)} 
\newcommand\res{\mathbb{F}} 
\newcommand\IMCres{\mathbb{I}}
\newcommand\idealres[1]{\IMCres^{J_{\ineq #1}}} 
\newcommand\colres[1]{\mathbb{F}^{\col{#1}}} 

\newcommand{\HSIdeal}[2]{\textrm{HS}_{#1}(#2)}
\newcommand\mdeg{\textrm{mdeg}}

\newcommand{\Cmatroid}{{$C$-matroidal }}
\newcommand{\Cmatroidal}{{$C$-matroidal }}

\begin{document}

\title{Focal matroids of covers and Homological properties of  matroids} 
\author[Paolo Mantero]{Paolo Mantero}

\address{University of Arkansas, Fayetteville, AR 72701}
\email{pmantero@uark.edu}
\thanks{The first author was partially supported
by Simons Foundation grant \#962192.}

\author[Vinh Nguyen]{Vinh Nguyen}
\address{University of Arkansas, Fayetteville, AR 72701}
\email{vinhn@uark.edu}

\begin{abstract} 
In this paper we prove that the Stanley--Reisner ideal or cover ideal $I$ of a matroid  is minimally resolvable by iterated mapping cones. As a technical tool for this purpose, we introduce and study focal matroids, which are submatroids of a matroid $\M$ that are constructed relative to minimal $\ell$-covers of $\M$.

Our second main result is that the monomial support of the multigraded Betti numbers of $I$ corresponds precisely to the squarefree minimal generators of the symbolic powers of $I$. In fact, we prove that matroidal ideals are the only squarefree ideals with this property, thus obtaining a new homological characterization of matroidal ideals.

These techniques  are foundational for a follow-up paper, where we will show that all symbolic power of $I$ are minimally resolvable by iterated mapping cones.

\end{abstract}

\maketitle

\section{Introduction}

Iterated mapping cones is a general inductive procedure to produce a free resolution of an ideal $I$ in a commutative Noetherian ring $R$. In the graded or local setting, we say that $I$ is {\em minimally resolvable} by iterated mapping cones if the resulting free resolution is minimal. 
In general, however, it is quite rare for resolutions by mapping cones to be minimal. For instance, given a monomial ideal $I$ in a polynomial ring over a field, the Taylor resolution provides a free resolution of $I$ \cite{Ta66}. It is well-known that Taylor resolutions are resolutions by iterated mapping cones, yet they are almost never minimal.  

On the positive side, monomial ideals with {\em linear quotients} are known to be minimally resolvable by iterated mapping cones, \cite[Lemma~1.5]{HT02} and \cite[Lemma~2.1]{JZ10}. They comprise almost all known classes of examples in the literature of ideals that are  minimally resolvable by iterated mapping cones. These examples include: 
\begin{enumerate}[(i)]
	\item (Eliahou--Kervaire, \cite{EK90}) strongly stable ideals;
\item (Aramova-Herzog-Hibi, \cite{AHH98}) squarefree lexsegment ideals;
	\item (Herzog--Takayama \cite{HT02}) ideals with regular decomposition functions. This includes the ideals in (i) and (ii), as well as {\em B-matroidal ideals}, i.e.  squarefree monomial ideals $I$ for which the supports of the minimal generators of $I$ are the basis of some matroid\footnote{This is what the authors meant by ``Stanley--Reisner ideal of a matroid" in \cite[Thm~1.10]{HT02}, see also the paragraphs before \cite[Lemma~1.3]{HT02}. In fact, this potential confusion is the main motivation for using $B$- and $C$- prefixes to distinguish different meanings for ``matroidal ideal".}; 
    \item (Conca--Herzog, \cite{CH03}) products of polymatroidal ideals;
	\item (Kokubo-Hibi \cite{KH06}, Mohammadi--Moradi \cite{MoMo10}) weakly polymatroidal ideals;
    \item (Mantero \cite{Ma20}) symbolic powers of star configurations.
\end{enumerate} 

A main goal of this paper is to prove the following theorem.
\begin{TheoremA}\label{thmA}(\Cref{Matroid-Resolution})
 The Stanley--Reisner ideal of any matroid is minimally resolvable by iterated mapping cones.
\end{TheoremA}
    
In our follow-up paper \cite{MN2}, we further prove that all symbolic powers of these ideals are minimally resolvable by iterated mapping cones.\smallskip

Being such a general procedure, there are three main challenges in writing an explicit minimal free resolution of $I$ obtained by iterated mapping cones. 
\begin{enumerate}
    \item The first one is developing a strong knowledge, independent of the inductive procedure, of the colon ideals and their resolutions.
    \item Secondly, the inductive procedure depends on a chosen ordering on the generators of $I$. Even if $I$ can be minimally resolvable by iterated mapping cone, one must find an ordering on the generators of $I$ which produces such  minimal resolution. 
    \item Lastly, an explicit description of the comparison map must be obtained at each step to give an explicit description of the differential maps in the free resolution.
\end{enumerate}

All the ideals in (i)--(iv) listed above have {\em linear quotients}. Recall that a monomial ideal $I\subseteq \kk[x_1,...,x_n]$ has linear quotient if there exists an ordering $f_1,\ldots,f_n$ on the set $G(I)$ of the minimal monomial generators of $I $ such that every colon ideal $(f_1,...,f_{n-1}):f_{n}$ is generated by a subset of $\{x_1,...,x_n\}$. Hence, the property of having linear quotients solves the first two challenges described above. The colon ideals are all monomial complete intersections, which are minimally resolvable by, for instance, the Koszul complex. The remaining challenge is to describe the comparison maps. This was done successfully in the following cases: (i)--(iii) above, 
and in \cite{Ga20} when $I$ is the Stanley--Reisner ideal of a uniform matroid. In general, however, even if $I$ has linear quotients, there is no known procedure to describe the comparison maps, or the differentials in the resolution of $I$. 

By (vi), if $I$ is the Stanley--Reisner ideal of a uniform matroid, then $\symp{I}$ has linear quotients. On the other hand, if $I =I_\M$ is the Stanley--Reisner ideal of any (independence complex of a) matroid $\M$, i.e. if $I$ is {\em C-matroidal}, then $\symp{I}$ does not necessarily have linear quotients. Despite this, we will show at the end of our two part paper that $\symp{I}$ is minimally resolvable by iterated mapping cone whenever $I$ is $C$-matroidal.  

To this end, we focus on investigating the cover ideal $J(\M)$ of any matroid $\M$. Indeed, by duality, $J(\M) = I_{\M^*}$, where $\M^*$ is the dual matroid of $\M$. Hence, results about $J(\M)$ for any matroid $\M$ can be applied to the dual matroid $\M^*$ to obtain results about $J(\M^*)=I_{\M^{**}} = I_{\M}$. Thus, for the following, we set $I := J(\M)$ for some matroid $\M$. 
\medskip

The key objects we employ to overcome challenge (1) in Theorem~A are the \textit{\cmname matroids} of $\M$, which we introduce and study in this paper. These are submatroids of $\M$ that are constructed relative to minimal covers of $\M$. For any simplicial complex $\Delta$ on $[n]$, an  $\ell$-cover of $\Delta$ is a function $\gamma:[n] \to \NN_{0}$ such that for any facet $F \subseteq [n]$ of $\Delta$, $\gamma(F) = \sum_{i \in F}\gamma(i) \geq \ell$. To any minimal $\ell$-cover $\gamma$ of $\Delta$ we associated the subcomplex $\Delta(\gamma) = \langle\{ F \in \F(\Delta) : \gamma(F) = \ell\} \rangle$ generated by facets where $\gamma$ obtains the smallest possible value. When $\Delta = \M$ is a matroid, we prove that $\cm \gamma$ is a submatroid of $\M$, which we call the {\em \cmname matroid} of $\M$ with respect to $\gamma$. The \cmname matroids of $\M$ play a crucial role in the resolution and Betti numbers of $I$. Indeed, one of our key technical result is an ordering on the minimal generators of $I$ so that the colon ideals are \Cmatroidal ideals. More precisely, we show they are cover ideals of co\cmname matroids of $\M$ (i.e. contractions of focal matroids), see \Cref{Colons-Are-CocriticalMatroids}. This addresses point (1) above. Therefore, a large part of our paper is dedicated to working out the technical structure of \cmname matroids, and the structure of their cover ideals. 

In regard to (2), we actually provide {\em many} different orderings which can be used. Indeed, \Cref{Con-Ordering} depends on a number of choices, and each of them gives an order on $G(J)$ which produces a minimal free resolution of $J$ by iterated mapping cones.

Leveraging the interpretation of the colon ideals as cover ideals of co\cmname matroids, we obtain one of our main results, Theorem~B, which provides the following intriguing connection: the multigraded shifts in any minimal resolution of any \Cmatroid ideal $J$ are precisely the squarefree minimal generators of the symbolic powers of $J$.  We employ Theorem~B in the proof of Theorem~A to obtain an efficient work-around of challenge (3). In fact, we actually show that the multi-graded shifts in homological degree $\ell$ in our resolution of $J$ by iterated mapping cones are precisely   the squarefree minimal generators of the $\ell$-th symbolic power of $J$. Since no minimal generator of $J^{(\ell)}$ can be a minimal generator of $J^{(\ell+1)}$, we immediately obtain that the resolution by iterated mapping cones is minimal, and the above-mentioned result:

\begin{TheoremB}\label{thmB}(\Cref{All-Degrees-Show-Up})
Let $J$ be any \Cmatroid ideal, and $\res_{\bullet}$ any multigraded minimal free resolution of $R/J$. Then $$G(\sfp_\ell(J)) = \mdeg(\res_\ell),$$
where $\mdeg(\res_\ell)$ is the set of monomials appearing as multigraded shifts in $\res_\ell$, and $\sfp_{\ell}(J)$ is the ideal generated by the squarefree monomials of $J^{(\ell)}$. 
\end{TheoremB}


The fact that the multigraded Betti numbers of a squarefree monomial ideal are the minimal generators of their symbolic powers is a curious phenomenon. In fact, we show that this property characterizes \Cmatroid ideals. Thus we obtain a new homological characterization of \Cmatroidal ideals. 

\begin{TheoremC}\label{thmC}(\Cref{Char-Matroid-100})
Let $J$ be a squarefree monomial ideal of $\h(J)\geq 2$. TFAE:
	\begin{enumerate}[(a)]
		\item $J$ is a \Cmatroid ideal;
		\item $\mdeg(\mathbb{F}_\ell) =G(\sfp_\ell(J))$ for some multigraded (not necessarily minimal) free resolution $\FF_\bullet$ of $R/J$, and for all $\ell\geq 1$;
				\item $\mdeg(\mathbb{F}_2) \subseteq G(\sfp_2(J))$ for some multigraded (not necessarily minimal) free resolution $\FF_\bullet$ of $R/J$.
	\end{enumerate} 
    
\end{TheoremC}

This new characterization involves $I^{(2)}$ instead of $I^{(3)}$, complementing the celebrated result by \cite{TT} (see also \cite{Var} and \cite{MT}) that $I$ is \Cmatroid ideal if and only if $I^{(3)}$ is Cohen-Macaulay if and only if $\symp{I}$ is Cohen-Macaulay for all $\ell \geq 1$. 

We remark that the minimal resolution of the Stanley--Reisner ideal of matroids, but not their symbolic powers, is already known by \cite{NPS}. We resolve $I$ anyway to demonstrate the potential application of our techniques in this first case. In our sequel paper, we will vastly generalize Theorem~B to connect the multigraded Betti numbers of any symbolic power $I^{(\ell)}$ to the minimal generators of higher symbolic powers of $I$.  Under a suitable ordering we show that the ideals of the \cmname matroids are also related to the colon ideals for any symbolic power $\symp{I}$. 

This paper is structured as follows. In \cref{Sec-2} we establish notation, recall a few facts about matroids and $C$-matroidal ideals, and present other basic results needed  throughout the paper.  

In \cref{Sec-3} we introduce \cmname matroids. We then explore their cover ideals and give a connection to symbolic powers. In \cref{Section-Ordering} we introduce and investigate orderings on $J(\M)$ used to prove our main results. We describe the colon ideals in these orderings and identify them as cover ideals of co\cmname matroids of $\M$. In \cref{Section-Resolution}, we combine the material from the previous sections to establish Theorems~A, B and C. 

\section{Preliminaries}\label{Sec-2}
  \subsection{Notation and basic facts}
We refer the reader to \cite{Oxley} and \cite{Welsh} for definitions and well-known facts about matroids, \cite{MS} 
for details about monomial ideals and simplicial complexes, and \cite{HHT} 
for results connecting covers of simplicial complex with monomial ideals and their symbolic powers. 

Throughout the paper we will adopt the following notation:
\begin{Notation}\label{Notation}
$\kk$ always denotes a field. We will always consider matroids or simplicial complexes over a finite vertex set, which we identify with $[n]$,  for some $n\geq 1$.

For any matroid $\M$ on $[n]$, we write:

\begin{itemize}
    \item $\B(\M)$, $\mc{I}(\M)$, and $\C(\M)$ for the sets of bases, independent sets, and circuits of $\M$, respectively;
    \item $r_\M(-)$, or simply $r(-)$ if $\M$ is understood, for the rank function of $\M$;
    \item $\M\dual$ for the dual matroid of $\M$;
    \item $\M^{\skel{i}}$ for the truncation of $\M$ to rank $i$;
    \item $\M|_A$, $\M\setminus A$, and $\M\con A$ for the matroid obtained from $\M$ by {\em  restriction} to $A$, {\em deletion} of $A$, and {\em contraction} along $A$, respectively, for any $A \subseteq [n]$;
    \item $U_{r,n}$ for the uniform matroid of rank $r$ on $[n]$.
\end{itemize} 
	
For any simplicial complex $\Delta$ on $[n]$, we write:
\begin{itemize}
\item $R:=\kk[x_1,\ldots,x_n]$ and $\m=(x_1,\ldots,x_n)$ for its homogeneous maximal ideal;
	\item $\F(\Delta)$ for the set of all facets of $\Delta$, and $\Delta = \langle F\,\mid\,F\in \F(\Delta)\rangle$;
	\item $F^*:=[n]-F$ for any $F\subseteq [n]$, and $\p_F:=(x_i\,\mid\,i\in F)\subseteq R$; 
	\item $J(\Delta):= \bigcap_{F \in \facet(\Delta)} \p_F\subseteq R$ for the {\em cover ideal} of $\Delta$;
		\item $I_\Delta := \bigcap_{F \in \facet(\Delta)} \p_{F^*}\subseteq R$ for the {\em Stanley--Reisner ideal} of $\Delta$;
	\item $I_\Delta^{(\ell)}:=\bigcap_{F\in \F(\Delta)}\left(\p_{F\dual}^\ell\right)$ for the {\em $\ell$-th symbolic power} of $I_\Delta$, and $\sfp_{\ell}(I):=SF(\symp{I})$ for the {\em squarefree part} of $\symp{I}$ (see below for the definition of $SF(-)$).

\item $\supp{M}=\{x_i\,\mid\,x_i \text{ divides }M \}$ for the support of any monomial $M$ in $R$. Often times, with a slight abuse of notation, we will identify $\supp{M}$ with $\{i\in [n]\,\mid\,x_i \text{ divides }M \}$. 
\item $x_F:=\prod_{i \in F} x_i$ for any $F\subseteq[n]$. Any squarefree monomial in $R$ has this form.
\end{itemize}

   For any monomial ideal $I$ in $R$, we adopt the following notation:
   \begin{itemize}
	\item $G(I)$ is the unique minimal generating set of $I$ consisting of monomials;
	\item $\sfp(I)$ is the ideal generated by all squarefree monomials in $I$;
	\item ${\rm Ass}(R/I)$ is the set of prime ideals $\p\subseteq R$ such that $\p=I:x$ for some $x\in R - I$; it is well--known that any such $\p$ has the form $\p=\p_F$ for some $F\subseteq [n]$;
	\item $\h\,I$ is the {\em height} of $I$, which is $\h\,I=\min\{|F|\,\mid\,\p_F\in \Ass(R/I)\}$; 
	\item $\dim(R/I)$ is the {\em dimension} of $R/I$, which is $n-\h\,I$;

\item	$\msupp{I}$ is the set $\bigcup_{M \in G(I)} \supp{M}$.
\end{itemize}

In this paper, we consider matroids as a subclass of (pure) simplicial complexes by identifying $\M$ with its independence complex.
\end{Notation}

To briefly illustrate part of Notation \ref{Notation}, consider $I=(x,y)^2\cap (x,z)^2\cap (y,z)^2$. Then $G(I)=\{xyz, x^2y^2, x^2z^2, y^2z^2\}$, $\Ass(R/I)=\{(x,y),\,(x,z),\,(y,z)\}$, $\h\,I=2$, $\sfp(I)=(xyz)$, $\supp{x^2y^2}=\{x,y\}$, and $\msupp{I} = \{x,y,z\}$.
\medskip

We now recall some well--known connections between generators of ideals associated to a matroid and the circuits and hyperplanes of the matroid. 

\begin{Proposition}\label{Basic-Matroid-Properties} 
	Let $\M$ be a matroid on $[n]$. Then 
\begin{enumerate}
    \item $\{F\subseteq [n]\,\mid\, x_F \in G(I_{\M}) \}$ is the set of circuits of $\M$.
    \item $\{F\subseteq [n]\,\mid\, x_F \in G(J(\M)) \}$ is the set of cocircuits of $\M$.
    \item $\{F\subseteq[n] \,\mid\, \p_F \in \Ass(R/J(\M))\} = \B(\M)$.
    \item $\{F\subseteq[n] \,\mid\, \p_F \in \Ass(R/I_\M)\} = \B(\M^*)$.
    \item $H$ is a hyperplane of $\M$ if and only if $x_{H\dual}$ is a minimal generator of $J(\M)$.
    \item $\dim(R/I_{\M})=r(\M)$, while $\h\,I_\M=\h\,J(\M\dual)=r(\M\dual)$. 
    \item $i\in [n]$ is a loop of $\M$ $\Llra$ $I_\M$ contains the variable $x_i$ $\Llra$ $J(\M)$ is extended from the smaller polynomial ring $\kk[x_j\,\mid\,j\neq i]$.
\end{enumerate}
\end{Proposition}

This paper is concerned with the Stanley-Reisner ideal or the cover ideal of a matroid $\M$. With regards to matroids, these ideals are dual to each other, $J(\M\dual) = I_{\M}$ and $J(\M) = I_{\M\dual}$, hence in \cite{MN26} we made a single definition.
\begin{Definition}\label{Def-Matroidal}
A squarefree monomial ideal $I\subseteq R$ is 
\term{\Cmatroid} if $I$ satisfies one of the following equivalent conditions:
	\begin{enumerate}
		\item  $I$ is the Stanley--Reisner ideal of a matroid,
		\item  $I$ is the cover ideal of a matroid.
	\end{enumerate}
\end{Definition} 
The prefix ``$C$-" informs the reader that the elements in $G(I)$ satisfy the {\underline{C}ircuit} axioms of matroids, since there are (at least) two different notions of ``matroidal ideals" in the literature. (see also \cite[Def.~2.10]{MN26}.)

\begin{Remark} By the above duality, most statements for \Cmatroidal ideals only need to be proved either for the Stanley--Reisner ideal of a matroid or for the cover ideal of a matroid. For symbolic powers of ideals it is more convenient to work with cover ideals of matroids, so this will be our approach in this paper.
\end{Remark}

  \subsection{Symbolic powers of \Cmatroid ideals}
The ultimate goal of this paper and our sequel paper is to describe the minimal graded free resolutions of symbolic powers of \Cmatroidal ideals. Here we recall some properties of their symbolic powers. Recall that they are ``homologically nice" in the sense that they are Cohen--Macaulay. Recall that an ideal $I\subseteq R$ is {\em Cohen--Macaulay} if its local cohomology modules $H_{\m}^i(R/I)$ are all zero except for $i=\dim(R/I)$ (equivalently, if the projective dimension ${\rm pd}(R/I)$ equals $\h\,I$).
It was proven by Minh and Trung \cite{MT}, and also  independently by Varbaro  \cite{Var}, that a simplicial complex $\Delta$ is (the independence complex of) a matroid $\Llra$ $I_{\Delta}^{(\ell)}$ is Cohen--Macaulay for every $\ell\in \ZZ_+$. See \cite{TT} and \cite{LM} for strengthenings of this result. 

The following theorem in \cite{MN26}, describing the structure of all symbolic powers of \Cmatroid ideals, is crucial for our results.

\begin{Theorem}(Structure Theorem, \cite[Thm~3.7]{MN26})\label{MatroidSymPowerThm} Let $I$ be a \Cmatroid ideal. Then
			{\small \begin{center}
					$G(\symp{I}) = \bigg\{$
					\begin{tabular}{l|l} 
						& $M_i \in \sfp_{c_i}(I)$, where $1 \leq c_i \leq \h\, I$\\
						$M=M_1 \cdots M_s$ 	 &\\
						&  with $\sum c_i = \ell$ and $\supp{M_1} \supseteq \ldots\supseteq \supp{M_s}$
					\end{tabular}
					$\bigg\}$.
			\end{center}} 
		\end{Theorem}

Since, by the Structure \cref{MatroidSymPowerThm}, the squarefree parts of symbolic powers play an important role, we record some useful facts about them.

\begin{Corollary}(\cite[Cor~3.15]{MN26})\label{Corr-Matroid-LCM}\label{Matroid-LCM-Corollary-2}\label{Matroid-LCM-Corollary} 
Let $I$ be a \Cmatroid ideal. Then the squarefree part of $\symp{I}$ is
	{\small $$\sfp_\ell(I)=(\LCM(M_1,...,M_\ell) \, | \, M_1,...,M_\ell \in G(I), \textrm{ and } M_i \nmid \; \LCM(M_1,...,M_{i-1}) \text{ for } 2 \leq i \leq \ell),$$}
			and 		$G(\sfp_\ell(I))$ consists of the minimal elements with respect to divisibility of the displayed set. 
			
			In particular, if $M_i \in \sfp_{\ell_i}(I)$ for $i=1,2$ and $\supp{M_1} \neq \supp{M_2}$, then $\LCM(M_1,M_2) \in \sfp_{\max\{\ell_1,\ell_2\}+1}(I)$.

		\end{Corollary}

	\begin{Remark}\label{Sqfree-parts}
		Let $J:=J(\Delta)$ for some pure simplicial complex $\Delta$. Then
		\begin{enumerate}
			\item $\sfp_a(J)$ is the cover ideal of the $(a-1)$-codimensional skeleton of $\Delta$. See e.g. \cite[Rem~3.18, Prop~3.19]{MN26}. In particular, if $\M$ is a matroid, then the skeletons of (the independence complex of) $\M$ coincide with the truncations of $\M$, hence $$\sfp_a(J(\M))=J(\M^{\skel{r(M)-a+1}}).$$
			\item Taking cover ideals of skeletons is ``additive" with a shift, that is $\sfp_{b}(\sfp_{a}(J)) = \sfp_{b+a-1}(J)$.
			\item If $b > \h\, J$, then $\sfp_b(J)$ is the cover ideal of the empty complex, hence $\sfp_b(J) = 0$. This is consistent with all results in this paper, so we need not assume $b \leq \h\, J$ when we consider $\sfp_b(J)$.
		\end{enumerate} 
	\end{Remark}

  \subsection{Standard Monomial Decomposition} For a \Cmatroid\,ideal $I$, the structure theorem states that the minimal generators of $\symp{I}$ have a ``tower" structure that resembles that of standard monomial theory. Throughout the paper, we manipulate these generators through their standard monomial decomposition. As such, here we collect, without proof, some elementary results from standard monomial theory.

\begin{Definition} For any monomial $M\in R$, we can decompose $M$ uniquely into a product $M = M_1\cdots M_s$, such that each $M_i$ is squarefree and such that $\supp{M_1} \sup ... \sup \supp{M_s}$. We say that $M_1\cdots M_s$ is the {\rm standard (monomial) form} 
	of $M$.
For convenience, we sometime allow $M_i = 1$, in which case the decomposition is unique only for the parts where $M_i \neq 1$.
\end{Definition}

\begin{Proposition}\label{SqFree-Decomp} Let $M$ and $N$ be monomials in $R$ with standard forms $M = M_1\cdots M_s$ and $N = N_1\cdots N_t$. By possibly setting some $M_i = 1$ or $N_i = 1$, we may assume $s = t$. Then 

\begin{enumerate}\label{SqFree-Decomp-Colon}
    \item The standard forms of $\LCM(M,N)$ and $\GCD(M,N)$ are $$\LCM(M,N) = \prod_{i=1}^{s}\LCM(M_i,N_i) \quad\;\; \GCD(M,N) = \prod_{i=1}^{s}\GCD(M_i,N_i).$$
    \item $N \,|\, M$ if and only if $N_i \,|\, M_i$ for all $1 \leq i \leq s$. In particular if $N \,|\, M$ then $(N/N_1) \, | \, (M/M_1)$.
    \item $N : M = \prod_{i=1}^{s} (N_i : M_i)$ -- note, however, that this may not be the standard form of $N : M$.
\end{enumerate}
    
\end{Proposition}

Recall that $(c_1,\ldots,c_s)\vdash \ell$ denotes a {\em partition} of $\ell\in \ZZ_+$, meaning that $c_1\geq c_2\geq \ldots \geq c_s\geq 1$ are integers with $\sum c_i=\ell$. 
\begin{Definition}\label{Def-SymbType}
			Let $I$ be a \Cmatroid ideal and $M\in G(\symp{I})$. Let $M = M_1\cdots M_s$ be the standard form of $M$, then Structure \cref{MatroidSymPowerThm} states that each $M_i$ is in $G(\sfp_{c_i}(I))$ for some partition $\underline{c}=(c_1,\ldots,c_s)\vdash \ell$ with $c_1 \leq \h\,I$. We call the partition $\underline{c}\vdash\ell$ the \term{symbolic type} of $M$. 
		\end{Definition}

As an easy consequence of the structure theorem, we derive a lemma allowing us to find minimal generators of $\symp{I}$, with a specific symbolic type, that divide a given monomial $M \in \symp{I}$.

\begin{Lemma}\label{Symbolic-Type-Divisibility} Let $I$ be a \Cmatroid\,ideal. Let $M$ be a monomial with standard form $M = M_1\cdots M_s$. Let $h_i=\max\{h\,\mid\,M_i\in I^{(h)}\}$ and $1\leq \ell\leq \sum h_i$. Choose any sub-partition (in the sense of Young tableaux) $\underline{c}=(c_1,\ldots,c_t)\vdash \ell$ of $(h_1,\ldots,h_s)$, i.e. $c_i \leq h_i$ for every $i$. 
Then, there exists $N \in G(\symp{I})$ with symbolic type $\underline{c}$ 
and standard form $N=N_1\cdots N_t$ where each $N_i \; |\; M_i$. 

\end{Lemma}

In fact, given any $M\in I^{(h)}$ and any integer $1\leq \ell \leq h$, one can employ the previous lemma to obtain all minimal monomial generators of $\symp{I}$ dividing $M$.

\subsection{Restriction, Contraction, and Colon Ideals}\label{Section-Colon-Ideals}

We collect a few results on how various matroid and ideal operations interact with each other. These statements are easy to check, so we will omit their proofs. 
\begin{Proposition}\label{Colon-Matroid-Structure} 
	Let $J=J(\M)$ be the cover ideal of a matroid $\M$, and let $N\notin J$ be a monomial. Then $J : N=J(\M\setminus \supp{N})$. In particular, if $J$ is any \Cmatroid ideal and $N\notin J$ is any monomial, then $J:N$ is \Cmatroid too. 
	 Furthermore, for any $\ell\geq 1$, $$\symp{(J : N)} = \symp{J} : N^\ell.$$
\end{Proposition}

\begin{Definition}\label{Def-Restriction} For any ideal $J\subseteq R=\kk[x_1,\ldots,x_n]$ and $A\subseteq [n]$, let $\kk[A]:=\kk[x_i\,\mid\,i\in A]$. Then the \term{restriction of $J$ to $A$} is $J|_A := J \cap \kk[A]$. 
    
\end{Definition}

We record the following facts about restrictions for future uses.

\begin{Remark}\label{Ideal-Restriction-Structure}
	Let $J\subseteq R=\kk[x_1,\ldots,x_n]$ be an ideal and $A\subseteq [n]$.
\begin{enumerate}

	\item  If $J=\q_1\cap \cdots \cap \q_t$ (e.g. the $\q_i$'s form a primary decomposition of $J$), then $J|_A=\bigcap \left(\q_i|_A\right)$. 
	
		\item 	If $J$ is monomial, then $J|_A$ is monomial and $J|_A=(N\in G(J)\,\mid\, \supp{N}\subseteq A)$.
		
		\item In particular, if $J=J(\Delta)$ for some simplicial complex $\Delta$, then $$\symp{(J|_A)} = \symp{J}|_A\qquad \qquad \text{ and }\qquad \qquad \sfp_\ell(J|_A) = \sfp_\ell(J)|_A.$$ 

	\item By (2) and \cref{Basic-Matroid-Properties} (1), $I_{\M|_A}=(I_\M)|_A$ holds for any matroid $\M$.

	 \item On the other hand, for any matroid $\M$, write $s:=\max\{|F\cap A|\,\mid\,F\in \B(\M)\}$ and $J:=J(\M)=\bigcap_{F\in \B(\M)} \p_F$. Then  $\h(J(\M|_A))=s$, and $$J(\M|_A)=\bigcap_{F\in \B(\M),|F\cap A|=s} \left(\p_F|_A\right)=\bigcap_{F\in \B(\M),|F\cap A|=s} \p_{(F\cap A)}.$$
	 In particular, if there is some basis $F\in \B(\M)$ with $F\subseteq A$, then it is also true that $$J(\M|_A)R=J(\M):_Rx_{A{\dual}}.$$

\item For any matroid $\M$, by duality, if $A\in \mc{I}(M)$ is an independent set, then 
$$J(\M/A)=
I_{M{\dual}|_{A{\dual}}}=\left(I_{M{\dual}}\right)|_{A{\dual}}=J(\M)|_{A{\dual}}.$$

\end{enumerate}	

\end{Remark}

\begin{Remark}\label{Rmk-2.15}
In contrast to $(4)$ above, the restriction of the cover ideal of a matroid is {\bf not} necessarily the cover ideal of the restriction.
In general, we only know $J(\M)|_A \subseteq J(\M|_A)$ and the containment can be strict -- see \Cref{examples-rmk}(1). 

However, equality holds true in the following special case that we will use, for instance, in the proof of \Cref{Critical-Matroid-Cover-Ideal-Decomposition}: If $\M=\M_1\oplus \M_2$ and $A$ is the vertex set of $\M_1$ then $J(\M_1)|_{A} = J(\M_1|_{A})$.  

Also, in general, restrictions and colons do not commute. For any ideal $N$, the containment\\ $J|_AR :_R N \subseteq (J :_R N)|_AR$ is true in general, but it can be strict. See \Cref{examples-rmk}(2).
\end{Remark}	

\begin{Example}\label{examples-rmk} 
We collect here a few concrete examples.\\
$($1$)$ let $J=J(U_{2,4})=(x_ix_hx_j\,\mid\,1\leq i < h<j\leq 4)$ and $A=\{1,2\}$. Then $J|_A=(0)\subsetneq J(U_{2,4}|_{A})=J(U_{2,2})=(x_1,x_2)\subseteq k[x_1,x_2]$. Nevertheless, $J(\M|_A)$ is \Cmatroid because $J(\M)|_{A}=\left(I_{M{\dual}}\right)|_{A}=I_{M{\dual}|_{A}}$, and it has $\h(J(\M)|_{A})=\min\{|F\cap A|\,\mid\,F\in \B(\M)\}$.\\
\\
$($2$)$ Let $J = J(U_{2,3})=(x_1x_2,x_1x_3,x_2x_3)\subseteq \kk[x_1,x_2,x_3]=R$, $A = \{x_1,x_2\}$ and $N = x_3$. Then $J|_AR : x_3 = (x_1x_2) : x_3  = (x_1x_2)$, but $(J : x_3)|_AR = (x_1,x_2)|_AR = (x_1,x_2)$.\\
\\
$($3$)$ As an example to illustrate $(6)$ above, consider the \Cmatroid ideal $J = (ad,ace,abe,bc,bde,cde)$. It is the cover ideal of $\M = \{\{a,b,c\}, \{a,b,d\}, \{a,b,e\},\{a,c,d\},\{a,c,e\},\{b,c,d\},\{b,d,e\},\{c,d,e\}\}$. Then $J|_{\{b,c,d,e\}} = (bc,bde,cde)$ is indeed the cover ideal of $\M/a = \{ \{b,c\},\{b,d\},\{b,e\},\{c,d\},\{c,e\}\}$.
   
\end{Example}

We isolate the following simple but useful result.
\begin{Lemma}\label{Sym-Powers-Contraction} 
Let $\M$ be a matroid and $A\in \mc{I}(M)$, then  $$G(\symp{J(\M\con A)}) =\{N\in G(\symp{J(\M)}))\,\mid\, \supp{N}\subseteq A\dual\} 
 \subseteq G(\symp{J(\M)})\quad  \text{ for all }\ell\geq 1.$$
\end{Lemma}

\begin{proof}

First, by \cref{Ideal-Restriction-Structure}$(6)$ followed by $(3)$,  $\symp{J(\M\con A)} =  \symp{(J(\M)|_{A\dual})} = \symp{J(\M)}|_{A\dual}\subseteq \symp{J(\M)}$. The result is now immediate from \cref{Ideal-Restriction-Structure}$(2)$\end{proof}
	
Let $v$ be an indepedent vertex of a matroid $\M$. Let $J:=J(\M)$ and $J\con v:=J(\M\con v)$. (see also \cref{Def-Iterated-Contraction}). From the above, we have the  partition $G(J)=G(J\con v) \sqcup [G(J)-G(J\con v)]$, which will be very relevant for our work, so we reserve a special notation for it.

\begin{Notation}\label{Notation-Div}
	Let $\M$ be a matroid, $J=J(\M)$, and let $v$ be an independent vertex of $\M$. Then, we set $J \div v $ to be the ideal generated by $G(J) - G(J\con v)$.
\end{Notation}

We record here the following immediate consequence of Lemma \ref{Sym-Powers-Contraction}. 

\begin{Corollary}\label{Partition-By-Contraction} 
Adopt \Cref{Notation-Div}. Then 
	$$
	G(J \con v) = \{ N \in G(J) : x_v \nmid N \}\qquad \text{ and }\qquad G(J \div v) = \{ N \in G(J) : x_v \; | \; N \}.$$ Furthermore, $G(J \con A) = \{ N \in G(J) : \; \supp{N} \cap A = \emptyset\} $, for any $A\in \mathcal I(\M)$.
\end{Corollary}

\section{\cmname matroids and their structure}\label{Sec-3}
In this section we introduce a notion which will be used to provide a combinatorial interpretation for the colon ideals appearing in the resolution by iterated mapping cones for symbolic powers of \Cmatroid ideals.
As \cmname matroids are combinatorial in nature, to describe them and their structure, we give priority to the language of $\ell$-covers, which we recall below. For results about ideals, including cover ideals of \cmname matroids, we will revert back to the language of monomial theory.

Let $\gamma : [n] \to \NN_0$, 
its {\em support } is $\supp{\gamma} := \{ i : \gamma(i) > 0\}$. Let $\Delta$ be a simplicial complex. For any $F \in \Delta$, we set $\gamma(F) := \sum_{i \in F} \gamma(i)$. For an integer $\ell$, we say that $\gamma$ is a {\em $\ell$-cover} of $\Delta$ if for all facets $F \in \facet(\Delta)$ $\gamma(F) \geq \ell$. 

We define a partial ordering on covers by comparing them pointwise, i.e. $\gamma \leq \delta$ if $\gamma(i) \leq \delta(i)$ for all $i \in [n]$. We write $\gamma < \delta$ if $\gamma \leq \delta$ and $\gamma \neq \delta$. A {\em basic} $\ell$-cover is a $\ell$-cover that is minimal with respect to this ordering. 

To any monomial $N=\prod_{i\in[n]} x_i^{a_i}$, we can bijectively associate a function $\gamma_N:[n]\lra \NN_0$ defined as $\gamma_N(i) = a_i$ for all $i$. Conversely, given $\gamma:[n]\lra \NN_0$, we set $N_\gamma=\prod_{i\in [n]}x_i^{\gamma(i)}$. 
The {\em order} of $\gamma$ at a face $F \in \Delta$ is $\gamma(F) = \sum_{i\in F} \gamma(i)$. Note that  the order of $\gamma$ is the same as the order of $N_\gamma$ in the $\p_F$-adic topology, i.e  $\max\{ h \mid N_{\gamma} \in \p_F^{h} \}$. 

One can check the following facts:
\begin{Remark}\label{Sufficient-Basic-Cover}
Let $N$ be a monomial, $\Delta$ a simplicial complex, and $J:=J(\Delta)$. \begin{enumerate}
	\item  $N\in \symp{J}$ $\Llra$ $\gamma_N$ is a $\ell$-cover of $\Delta$ $\Llra$ $\gamma_N(F)\geq \ell$ for all $F\in \F(\Delta)$;
	\item $N\in G(\symp{J})$ $\Llra$ $\gamma_N$ is a basic $\ell$-cover of $\Delta$  $\Llra$ $\gamma_N$ is a $\ell$-cover of $\Delta$ and $\forall\,i \in \supp{\gamma_N}$ there exists $F \in \facet(\Delta)$ with $ i \in F$ and $\gamma_N(F) = \ell$.
\end{enumerate}
\end{Remark}

We now present a fundamental notion for this paper.
\begin{Definition}\label{Def-FocalMatroid} Let $\gamma$ be a $\ell$-cover of a simplicial complex $\Delta$. A facet $F\in \facet(\Delta)$ is a \term{\cmname facet} of $\gamma$ or, equivalently, $\p_F$ is a \term{\cmname prime} of $\gamma$, if $\gamma(F) = \ell$. 
The \term{\cmname complex} of $\gamma$ is the subcomplex $$\Delta(\gamma) := \langle \{ F \in \facet(\Delta) \,:\, \gamma(F) = \ell\} \rangle.$$ 
When $\Delta = \M$ is a matroid we will show later that $\cm{\gamma}$ is a submatroid of $\M$ (see \cref{Critical-Matroid}), in which case we call $\cm{\gamma}$ the \term{\cmname matroid} of $\gamma$ (in $\M$) and $F$ a \term{\cmname basis} of $\gamma$ (in $\M$).

\end{Definition}

In the remaining results of this section, we describe the structure of the \cmname complex for matroids. For instance, \Cref{Varbaro-Argument} provides useful restrictions to the exchanges one can make starting from a focal basis $F$. Note that \Cref{Varbaro-Argument}(4) is already present in the proof of \cite[Thm~2.1]{Var}.

\begin{Lemma}\label{Varbaro-Argument}
Let $\M$ be a matroid and $\gamma$ an $\ell$-cover of $\M$. Let $F$ be a \cmname basis of $\gamma$ and let $G\in \B(\M)$, then we have the following statements about the exchange properties of $\M$. 
\begin{enumerate}
  \item From the (multibasis) exchange property, for any $A \subseteq F - G$ there is a $B \subseteq G - F$ such that $(F-A)\cup B \in \B(\M)$. Then, for any such $B$, we have $\gamma(A) \leq \gamma(B)$, and if $A \subseteq \supp \gamma$ then $B \subseteq \supp \gamma$.

    \item From the bijective basis exchange property \cite[Thm 1]{Brualdi1969}, there is a bijection $\sigma : (F-G) \to (G-F)$ such that for every $i \in F-G$, $(F-i)\cup \sigma(i)\in\B(\M)$. Then for all $i \in F -G$, $\gamma(i)\leq \gamma(\sigma(i))$. 
  
\end{enumerate}
Now, suppose further that $G$ is a \cmname basis too. 

\begin{enumerate}
    \item[(3)] In $(1)$ above, assume in addition that symmetrically $(G-B)\cup A$ is a basis. Then $\gamma(A) = \gamma(B)$ and $ A \subseteq \supp \gamma$ if and only if $B \subseteq \supp \gamma$.
    \item[(4)] With $\sigma$ as in $(2)$ above, we also have $\gamma(i) = \gamma(\sigma(i))$.
\end{enumerate}

\end{Lemma}

\begin{proof} We prove the inequalities in $(1)$ and $(2)$ first.
Let $F' := (F-A) \cup B$. Since $F \cap B=\emptyset$, and $F'$ is a basis, the inequality in $(1)$ follows from 
$$\ell \leq \gamma(F') = \gamma(F) - \gamma(A) + \gamma(B) = \ell - \gamma(A) + \gamma(B).$$
$(2)$ follows similarly, just replace $A$ and $B$ in the expression above with $\{i\}$ and $\{\sigma(i)\}$, respectively. 

For the statement about supports in $(1)$, notice that $F - F' = A$ and $F' - F = B$, so by $(2)$ applied to $F$ and $F'$ there is a bijection $\sigma:A \to B$ such that for all $a \in A$, $\gamma(a) \leq \gamma(\sigma(a))$. Hence assuming $A \subseteq \supp \gamma$, for any $b \in B$, $\gamma(b) \geq \gamma(\sigma^{-1}(b))>0$. Thus $B \subseteq \supp \gamma$.

$(3)$ follows from a symmetric argument using $(1)$. 

$(4)$ By  $(2)$, we know that for every $i \in F -G$, $\gamma(\sigma(i)) \geq \gamma(i)$. Since $\gamma(F) = \gamma(G)$, we also have $\sum_{i\in F-G}\gamma(\sigma(i))= \sum_{i\in F-G}\gamma(i) $. Hence the term-wise inequality implies $\gamma(\sigma(i)) = \gamma(i)$. 
\end{proof}

Note that the lemma can be used for any $\ell$-cover $\gamma$. If $\gamma$ is basic, by \cref{Sufficient-Basic-Cover}, a focal basis is guaranteed to exist, while if $\gamma$ is not basic, a focal basis may not exist.

We record the following result about a covering property of focal facets of $\Delta$ for future use.
\begin{Proposition}\label{Critical-Primes-Cover} 
	Let $\gamma$ be a basic $\ell$-cover of a simplicial complex $\Delta$. Then,
	\begin{enumerate}
		\item $\supp{\gamma} \sub \cup_{F \in \Delta{(\gamma)}} F$;
		\item if $\Delta$ is either a matroid or a graph, then $\{i\in [n]\,\mid\,\{i\}\in \Delta\}
		\sub \cup_{F \in \Delta(\gamma)} F$. 
	\end{enumerate} 
\end{Proposition}

\begin{proof} (1) Assume not, then $i \notin \cup_{F \in \Delta(\gamma)} F$ for some $i \in \supp{\gamma}$. Let $\gamma':=\gamma - e_i$, where $e_i(h)=\delta_{ih}$, Kronecker's delta. By construction $\gamma' < \gamma$ and it is easily seen that $\gamma'$ is a $\ell$-cover, contradicting the minimality of $\gamma$.

	(2) Let $i\in \Delta$ and let $G\in\F(\Delta)$ with $i \in G$. By (1), we may assume $i \notin \supp {\gamma}$.\\
	\indent First, assume $\Delta$ is a matroid. Let $F \in \facet( \Delta(\gamma))$. If $i\in F$ we are done. If not, since $\Delta$ is a matroid, there exists a $j \in F-G$ such that $F' = (F - j ) \cup i\in \F(\Delta)$. To conclude we show $F'\in  \Delta(\gamma)$. 
	Since $i \notin \supp {\gamma}$, by \cref{Varbaro-Argument}(1) we find that $j\notin \supp{\gamma}$, so $\gamma(F') = \gamma(F) = \ell$. Thus $F' \in \facet( \Delta(\gamma))$.
	
	Now, assume $\Delta$ is a graph. We show $\gamma(G) = \ell$, so that $G\in \F(\Delta(\gamma))$. Write $G = \{i , j\}$ for some $j \neq i$, then since $i \notin \supp{\gamma}$ we have $\gamma(G) = \gamma(j)$. Suppose $\gamma(j) > \ell$. Then $\gamma'=\gamma-e_j$ is a $\ell$-cover, contradicting the assumption that $\gamma$ is basic. 
\end{proof}

\begin{Example}
There exist pure, connected, 2-dimensional simplicial complexes $\Delta$ for which the conclusion of \cref{Critical-Primes-Cover}(2) does not hold. E.g. $\gamma = (\gamma(1), \gamma(2), ..., \gamma(7)) = (0,1,1,0,0,0,0)$ is a basic $1$-cover of $$\Delta = \langle\{1,2,3\}, \{2,4,5\}, \{3,6,7\} \rangle,$$
	but $\Delta(\gamma) = \{\{2,4,5\}, \{3,$ $6,7\}\}$ does not contain the vertex $1$. However, consistent with \cref{Critical-Primes-Cover}(1), we can see that $\Delta(\gamma)$ indeed covers $\supp{ \gamma} = \{2,3\}$.
	
\end{Example}

\cref{Varbaro-Argument}(3) shows that the symmetric basis exchange property involving $F,G \in \B(\M)$ with $\gamma(F)=\gamma(G) = \ell$ can be done in $\cm{\gamma}$, yielding the following result.

\begin{Corollary}\label{Critical-Matroid} Let $\M$ be a matroid and $\gamma$ a basic $\ell$-cover of $\M$. Then $\cm{\gamma}$ is a submatroid of $\M$ with $r(\cm{\gamma})=r(\M)$. We call $\cm{\gamma}$ the \cmname matroid of $\gamma$ (in $\M$).
\end{Corollary}

Next, we describe the structure of all \cmname matroids. 

\begin{Theorem}\label{Critical-Matroid-Structure} Let $\gamma$ be a basic $\ell$-cover of a matroid $\M$ and let $N_\gamma = N_1\cdots N_s$ be the standard form of $N_\gamma$. (see Structure \cref{MatroidSymPowerThm}.) 
Set $\gamma_i:=\gamma_{N_i}$ for every $1 \leq i \leq s$, $\gamma_{s+1}:=0$, and  $\gamma_0:=\gamma_{x_{[n]}}$.	Then, 
$$\cm{\gamma} = \bigoplus_{i=0}^{s} \cm{\gamma}|_{\supp{\gamma_i}-\supp{\gamma_{i+1}}}.
$$
Hence, we may split up the summands as $\cm{\gamma}=\zcm{\gamma}\oplus \ccm{\gamma}$, where
$$\zcm{\gamma}:=\cm{\gamma}|_{\supp{\gamma}{\dual}}\;\; \text{and}\;\; \ccm{\gamma}:= \cm{\gamma}|_{\supp{\gamma}}=
\bigoplus_{i=1}^{s} \cm{\gamma}|_{\supp{\gamma_i}-\supp{\gamma_{i+1}}} .$$

In particular, for any $F \in \B(\ccm{\gamma})$, we have $F \subseteq \supp{\gamma}$ and $\gamma(F)= \ell>0$, and $\gamma(G)=0$ for any $G\in \B(\zcm{\gamma})$.

\end{Theorem}

\begin{proof} 
For $0 \leq i \leq s$, let $A_i:= \supp{\gamma_i} - \supp{\gamma_{i+1}}$. 
From the definition of the $\gamma_i$'s and Structure \cref{MatroidSymPowerThm}, one has that $A_i=\gamma^{-1}(i)$. 

We consider the partition $[n]=\bigsqcup_{i=0}^s A_i$. \cref{Varbaro-Argument}(4) implies that, for any $i\geq 0$, the set $\{F\cap A_i\,\mid\,F\in \B(\cm{\gamma})\}$ is the set of basis of a submatroid $\M_i$ of $\cm{\gamma}$. Therefore, (inductive use of) \cite[Prop~4.5]{MN26} yields the decomposition $\cm{\gamma} = \bigoplus_{i=0}^{s}\M_i$ where each 
$\M_i=\cm{\gamma}|_{A_i}$. 
\end{proof}

\begin{Remark}	
By \cref{Critical-Matroid-Structure}, $\cm{\gamma}$ is always decomposable, except possibly when $\gamma$ is a constant function, i.e. there is $a\in \ZZ_+$ such that $\gamma(i)=a$ for all $i$ with $x_i\in \msupp{J(\M)}$. In this case $\cm{\gamma} = \M$ and, by \cite[Thm~4.7]{MN26}, the direct summands of $\cm{\gamma}$ determine the symbolic Noether number of $J(\M)$.
\end{Remark}

We can now easily deduce a characterization of $\zcm{\gamma}$ in terms of some minors of $\M(\gamma)$.
\begin{Proposition}\label{Bijection-Along-Cover} Let $\M$ be a matroid. Let $\gamma$ be a basic $\ell$-cover of $\M$, let $H$ be any independent set of $\cm{\gamma}$ with $H\subseteq \supp{\gamma}$ and $\gamma(H) = \ell$. Then
	\begin{enumerate}

		\item $\cm{\gamma}/H = \zcm{\gamma}.$

        \item If, additionally, $\gamma$ is squarefree, i.e. $\gamma(i) \leq 1$ for all $i$ (e.g. if $\ell=1$), then 
        $$
        \M(\gamma)|_{\supp{\gamma}\dual} = \M|_{\supp{\gamma}\dual}.
        $$

	\end{enumerate} 
\end{Proposition}

Therefore, the matroid $\cm{\gamma}/H$ is an invariant of $\M$ and $\gamma$ (it is independent of $H$). So, we provide the following definition.
\begin{Definition}
Let $\gamma$ be a basic $\ell$-cover of a matroid $\M$. The \term{co\cmname matroid} of $\gamma$ is $\zcm{\gamma}:=\cm\gamma|_{\supp{\gamma}\dual}$ or, equivalently, $\cm{\gamma}/H$ for some (equivalently, every) $H \in \mathcal I(M)$ with $H\subseteq \supp{\gamma}$ and $\gamma(H)=\ell.$
\end{Definition}

\begin{proof} (1) Notice that $H$ is a basis of $\ccm{\gamma}$. The statement then follows from the decomposition $\cm{\gamma} = \zcm{\gamma} \oplus \ccm{\gamma}$ in \cref{Critical-Matroid-Structure}.

(2) Since $\gamma$ is a squarefree $\ell$-cover, $r(\M|_{\supp{\gamma}\dual})=r(\M)-\ell$. Also, for any $H \in \supp{\gamma}$ with $\gamma(H) = \ell$ we must have $|H| = \ell$. Thus by $(1)$, we have $r(\M) - \ell =r(\cm{\gamma}\con H)=r((\cm{\gamma}|_{\supp{\gamma}\dual})$. Thus $r(\M|_{\supp{\gamma}\dual}) = r((\cm{\gamma}|_{\supp{\gamma}\dual})$.
Since $\B(\cm{\gamma}) \subseteq \B(\M)$, from the equality of ranks, it follows that $\mathcal B(\cm{\gamma}|_{\supp{\gamma}\dual})\subseteq \mathcal B (\M|_{\supp{\gamma}\dual})$. It remains to show the reverse containment. Let $F\in \mathcal B (\M|_{\supp{\gamma}\dual})$, then $F\in \mathcal I(\M)$ with $|F|=r(\M)-\ell$ and $\gamma(F)=0$. Let $G\supsetneq F$ be a basis of $\M$, so $\gamma(G)\geq \ell=|G-F|$. Since $\gamma$ is squarefree, we see that equality holds, yielding  $G=F\sqcup K$ for some $K\subseteq \supp{\gamma}$ with $\gamma(K)=\ell$. Hence $F\in \cm{\gamma}/K= \cm{\gamma}|_{\supp{\gamma}\dual}$, where the equality follows from $(1)$.
\end{proof}

We can now easily compute the heights of the cover ideals of the matroids in  \cref{Critical-Matroid-Structure}.

\begin{Corollary}\label{Critical-Matroid-Height}
Let $J=J(\M)$ be the cover ideal of a matroid $\M$, let $N\in G(\symp{J})$ with symbolic type $(c_1,...,c_s)$. (see \cref{Def-SymbType}.) Then, 
$$
\h\, J(\cm{\gamma_N}) = r(\M)=\h\,J, \quad  \h\, J(\ccm{\gamma_N}) = c_1, \quad \h\, J(\zcm{\gamma_N}) = \h\,J - c_1.
$$
\end{Corollary}

\begin{proof} By \cref	{Critical-Matroid} $r(\cm{\gamma_N})=r(\M)$, so one has $\h\, J(\cm{\gamma_N}) = r(\M)$. 
By \cref{Critical-Matroid-Structure} and \cite[Thm~4.5]{MN26}, $J(\cm{\gamma_N})=J(\zcm{\gamma_N}) + J(\ccm{\gamma_N})$ is a sum of transversal ideals, so we only need to show $\h\, J(\ccm{\gamma_N}) = c_1$. By \cref{Critical-Matroid-Structure} $r(\ccm{\gamma_N}) = r(\ccm{\gamma_{\sqrt{N}}})$. Now $\sqrt{N}$ is a minimal generator of $J^{(c_1)}$. Since $\sqrt{N}$ is squarefree and every basis $F$ of $ \ccm{\gamma_{\sqrt{N}}}$ is contained in $\supp{\sqrt{N}}$, we have $|F| = 
O_F(\sqrt{N}) = c_1$.
\end{proof}

We now give an alternative, useful description for the symbolic powers of the cover ideal of any \cmname matroid in terms of a colon ideal of a symbolic power of the cover ideal of the matroid.

\begin{Proposition}\label{Critical-Matroid-Cover-Ideal} 
Let $J=J(\M)$ be the cover ideal of a matroid $\M$, and let $N\in G(\symp{J})$, then for any $k\geq 1$,$$
J(\cm{\gamma_N}^{(k)}){\color{blue}} = J^{(\ell+k)} : N.
$$
Furthermore, $\{LN\,\mid\, L \in G(J(\cm{\gamma_N}^{(k)})\}\subseteq  G(J(\cm{\gamma_N})^{(\ell+k)})\cap G(J^{(\ell+k)})$.

\end{Proposition}

\begin{proof} 
For any $F\in \B(\M)$, one has $O_F(N)=\ell$ $\Llra$ $F\in \B(\cm{\gamma_N})$.  
Hence, a monomial $L$ is in $J(\cm{\gamma_N})$ $\Llra$ $O_F(LN)\geq (\ell+k)$ for all $F\in \B(\cm{\gamma_N})$. So, $J(\cm{\gamma_N}) = J^{((\ell+k)}) : N$. 

Next, take any $L \in G(J(\cm{\gamma_N}))$. We will use \cref{Sufficient-Basic-Cover}(2) to show the inclusion in the ``furthermore" statement. Fix $x_i\in \supp{LN}$. By minimality of the generators $N$ and $L$, $\{i\}$ is  independent in $\M$. Then, by \cref{Critical-Primes-Cover}(2), there is an $F \in \B(\cm{\gamma_N})\subseteq \B(\M)$ with  $i \in F$ and, by \Cref{Sufficient-Basic-Cover}(2), we can find such an $F$ with  $O_F(L) = k$. Then $O_F(LN)=O_F(L)+O_F(N)=\ell+k$.  By the first paragraph, we have $LN \in J^{(\ell+k)}$ and $LN \in J(\cm{\gamma_N})^{(\ell+k)}$. Hence with the above and  \cref{Sufficient-Basic-Cover}(2) we have simultaneously shown that $LN\in G(J(\cm{\gamma_N})^{(\ell+k)})$ and $LN\in G(J^{(\ell+k)})$. \end{proof}

We conclude this section by identifying the \cmname matroids as matroids associated to certain colons of squarefree monomial ideals. These ideals are strongly connected to the ones appearing in \cref{Section-Ordering}, see \Cref{Colons-Are-CocriticalMatroids}, 
and they play a role in the description of the resolution of the symbolic powers of \Cmatroidal ideals, \cite{MN2}.

We need some technical lemmas about skeletons and contraction. First, we give a setting where taking cover ideals of truncations commute with the colon operation.

\begin{Proposition}\label{Colon-Skeleton-Commute}
	 Let $J=J(\M)$ be the cover ideal of a matroid $\M$. Let $N$ be a monomial. 
     \begin{enumerate}
         \item If $N\notin J$, then for $\ell \geq 1$ $$\sfp_\ell(J: N) = \sfp_\ell(J) : N.$$
         \item If $N\in G(J)$, then for any $x_v\in \supp{N}$ and  any $\ell > 1$ 
	 $$\sfp_\ell(J) : N = \sfp_{\ell-1}(J(\M\con v) : N).$$
     \end{enumerate} 
\end{Proposition}

\begin{proof} Let $A:=\{i\in [n]\,\mid\,x_i\in \supp{N}\}$. \\
(1) Combining \cref{Ideal-Restriction-Structure}(5) and \cref{Sqfree-parts}(1), the matroids corresponding to the left hand and right hand side are respectively $$(\M|_{A\dual})\skel{r(\M)-\ell+1} \textrm{\quad and \quad} (\M\skel{r(\M)-\ell +1})|_{A\dual}.$$ Since $N\notin J$, we have $r(\M)-r(A\dual) = 0$, so by \cite[Prop 5.1]{Crapo-Schmitt-05} the two matroids are equal.\smallskip

(2) Since $N\in G(J)$, we have $r(M) - r(A\dual) = 1$. Now, the matroid of $\sfp_\ell(J) : N$ is $(\M\skel{r(\M)-\ell +1})|_{A\dual}$ which, by \cite[Prop 5.1]{Crapo-Schmitt-05}, is equal to $(\M|_{A\dual})\skel{r(\M)-\ell }$. It remains to show that $\M|_{A\dual} = (\M\con v)|_{A\dual}$, but this is \cref{Bijection-Along-Cover}(2) applied to $\gamma_N$.
\end{proof}

We now present a result on the finer structure of the ideal $J(\cm \gamma)$. We will decompose the ideal into a sum of monomial ideals with disjoint supports, and the summands are all colon ideals of a specific form.
\begin{Proposition}\label{Critical-Matroid-Cover-Ideal-Decomposition} 
	Let $J:=J(\M)$ be the cover ideal of a matroid and $N \in G(\symp{J})$. Write $N= N_1 \cdots N_s$ as in the Structure \cref{MatroidSymPowerThm}, with symbolic type $(c_1,...,c_s)$. Set $N_0 = x_{[n]}$,  $N_{s+1} = 1$, and $c_{s+1}=0$. Then, 
	$$J(\cm{\gamma_N}) = \sum_{i=1}^{s+1} \sfp_{c_i+1}(J|_{\supp{N_{i-1}}}) : N_i.$$
Thus, for $1\leq i \leq s + 1$, letting $A_{i-1}:=\{h\in [n]\,\mid\,x_h\in \supp{N_{i-1}}-\supp{N_{i}}\}$,  we have
	 $$J(\cm{\gamma_N}|_{A_{i-1}}) = \sfp_{c_i+1}(J|_{\supp{N_{i-1}}}) : N_i.$$ 
\end{Proposition}

\begin{proof} We begin with the first equality.\\
``$\subseteq"$ By \cref{Critical-Matroid-Cover-Ideal} we know that $J(\cm{\gamma_N}) = J^{(\ell+1)} : N$. Let $N' \in G(J^{(\ell+1)})$ with $N' \neq N$. Write $N' = N'_1\cdots N'_t$, as in the Structure \cref{MatroidSymPowerThm}. Since $N \neq N'$, there exists a minimal index $1 \leq k \leq s+1$ such that $\supp{N'_{k}} \neq \supp{N_{k}}$. 
To prove the desired inclusion, we show that $N':N\in \sfp_{c_k + 1}(J|_{\supp{N_{k-1}}}) : N_k$.

By \cref{Corr-Matroid-LCM} $\LCM(N'_k,N_k) \in \sfp_{c_k +1}(J)$. Consider the following monomial $N''$, where we replace  the term $N_k$ in $N$ with $\LCM(N'_k,N_k)$, 
$$N'' := (N_1 \cdots N_{k-1})\LCM(N'_k,N_k) (N_{k+1}\cdots N_s) \in J^{(\ell+1)}.$$ 

By choice of $k$, and the nesting of the supports, we have $\supp{\LCM(N_k',N_k)} \subseteq \supp{N_{k-1}}$. Hence the displayed expression for $N''$ is a squarefree monomial decomposition for $N''$. We now apply \cref{Symbolic-Type-Divisibility} to $N''$, by viewing it in $J^{(\ell+1)}$ 
to obtain a minimal generator $L = L_1\cdots L_s\in G(\symp{J})$ with type $(c_1,...,c_k +1, ...,c_s-1)$ such that $L_i \; | \; N''_i$ for all $i$. In particular, $L_{k-1} \mid N''_{k-1} =N_{k-1}$ and they are both minimal generators of $\sfp_{c_{k-1}}(J)$ hence $L_{k-1}= N_{k-1}$. Then we have $\supp{L_k} \subseteq \supp{L_{k-1}} = \supp{N_{k-1}}$. By \cref{Ideal-Restriction-Structure}(3) $L_k \in \sfp_{c_k + 1}(J)|_{\supp{N_{k-1}}} = \sfp_{c_k + 1}(J|_{\supp{N_{k-1}}})$. Now, $L_k \; | \;  \LCM(N'_k,N_k)$, and note that $\LCM(N'_k,N_k) : N_k = N'_k : N_k$. Hence we deduce the following divisibility chain
$$(L_k : N_k) \; | \; (N'_k : N_k) \; | \; (N' : N).
$$ 
We conclude that $N': N$ is divisible by $L_k : N_k$, which is in $\sfp_{c_k + 1}(J|_{\supp{N_{k-1}}}) : N_k$.

``$\supseteq$" Let $L\in G(\sfp_{c_i + 1}(J|_{\supp{N_{i-1}}}) : N_i)$, we show $LN \in J^{(\ell+1)}$. Note that the ideals in the sum summation are pairwise disjoint, because 
	\bee\label{supp}
	\supp{\sfp_{c_i+1}(J|_{\supp{ N_{i-1}}}) : N_i} = \supp{N_{i-1}} \setminus \supp{N_i}.
	\eee 
	By \cref{Ideal-Restriction-Structure}(3) $LN_i \in \sfp_{c_i+1}(J|_{\supp{N_{i-1}}}) =  \sfp_{c_i+1}(J)|_{\supp{N_{i-1}}}$, so by the structure \cref{MatroidSymPowerThm} the monomial $LN = N_1\cdots (LN_i) \cdots N_s\in G(J^{(\ell+1)})$.
	
	Finally, from the first part of the statement and equality (\ref{supp}), we have 
$$\begin{array}{ll}
J(\cm{\gamma_N}|_{A_{i-1}}) & =	J(\cm{\gamma_N})|_{\supp{N_{i-1}} - \supp{N_i}}\\
& = \left(\sfp_{c_i+1}(J|_{\supp{N_{i-1}}}) : N_i\right)|_{\supp{N_{i-1}} - \supp{N_i}}\\
& =\sfp_{c_i+1}(J|_{\supp{N_{i-1}}}) : N_i. 
\end{array}$$

Note that for the first equality where we commute taking cover ideal and restriction we are using \cref{Rmk-2.15} and the decomposition of $\cm{\gamma_N}$ of \Cref{Critical-Matroid-Structure}.
\end{proof}

\begin{Example}\label{Example-Focal-Matroid}
Let $\M$ be a matroid and $J=J(\M)$. We provide an example of using \cref{Critical-Matroid-Cover-Ideal-Decomposition} to describe $\cm{\gamma_N}$, for any $N\in G(J)$.   We will use the notation for $A_i$ in \cref{Critical-Matroid-Cover-Ideal-Decomposition}, and note that $s = 1$. We have the following description of the cover ideals of the matroids $\cm {\gamma_N} = \ccm {\gamma_N} \oplus \zcm {\gamma_N}$:

\begin{enumerate}
	\item $J(\zcm{\gamma_N})$ is the term where $i=1$, which is $\sfp_2(J) : N$;
	\item $J(\ccm{\gamma_N})$ is the sum of the terms for $1 < i \leq s+1$. In this case, because $s=1$, there is only one term, which is $J|_{\supp{N}} =(N)$; 
	\item Hence, $J(\cm{\gamma_N}) =J(\zcm{\gamma_N})+J(\ccm{\gamma_N}) = (\sfp_2(J) : N,\,N)$.
\end{enumerate} 
    
\end{Example}

We will see in the next section that, under an appropriate ordering, the colon ideals appearing in the iterated mapping cones of a \Cmatroid ideal have the form   $J(\zcm{\gamma_N})$, for some matroid $\M$.

\section{Iterated Contractions and Iterated Mapping Cones}\label{Section-Ordering}
In this section we introduce and study orderings allowing us to minimally resolve $J(\M)$ by iterated mapping cones. \Cref{Con-Ordering} provides a large number of different orderings. For any choice of a basis $B$ of the matroid $\M$, any choice of an order $v_1,\ldots,v_c$ on the vertices of $B$, any choices of orders on the sets of co-circuits of $\M\con (v_1,\ldots,v_i)$ which are not co-circuits of $\M\con (v_1,\ldots,v_{i+1})$, we obtain an ordering on $G(J)$ which we can use to minimally resolve $J$ by iterated mapping cones. 

We first recall the meaning of ``resolving ideals by iterated mapping cones".
\begin{Construction}(See also \cite[Construction~27.3]{Pe11}) Given an ideal $I\subseteq R$, fix an ordered, finite generating set $N_1,\ldots,N_t$ of $I$. 
For every $1 \leq j \leq t$, we construct a projective resolution $\FF_j$ for $I_j := (N_1,...,N_j)$ by fixing any projective resolution of $R/(I_{j-1}:N_j)$, taking the inductively constructed projective resolution of $R/I_{j-1}$, and applying the Mapping Cone (with respect to these two resolutions) to the short exact sequence 
$$0 \to R/(I_{j-1} : N_j) \xrightarrow{\cdot{N_j}} R/(I_{j-1}) \to R/I_j \to 0.$$
We say $\FF:=\FF_t$ is a {\em projective resolution by iterated mapping cones} of $R/I_t=R/I$.
\end{Construction}

Clearly, any ideal can be resolved by iterated mapping cones. However, in the homogeneous (or local) case, the resolution obtained may not be minimal, even if one employs a minimal projective resolution of $R/(I_{j-1}:N_j)$ for every $j$. The Taylor Resolution is a well--known example of a resolution by iterated mapping cones, and it is rarely minimal. To our knowledge, the majority of ideals that are known to be minimally resolved by iterated mapping cones are {\em ideals with linear quotients}. Recall that a monomial ideal $I \sub R$ has {\em linear quotients} if there exists an ordering $N_1,\ldots,N_t$ on $G(I)$ such that every  colon ideal $I_{j-1} : N_j$ is generated by a subset of variables. \cite[Lemma~1.5]{HT02} along with  \cite[Lemma~2.1]{JZ10} proved that these ideals are minimally resolved by mapping cones. Dochtermann and Mohammadi proved that, if additionally $I$ has a regular decomposition function, then the minimal resolution by iterated mapping cones is cellular \cite{DM14}.
\begin{Example}
The ideals (i)--(vi) in the Introduction have linear quotients.
\end{Example} 

Since symbolic powers of ideals associated to uniform matroids have linear quotients (\cite{Ma20}), one may wonder if \Cmatroid ideals, and, even more optimistically, their symbolic powers, have linear quotients. The answer is a resounding {\em no}, e.g. the Fano plane gives a negative example for any $\ell\geq 1$. 
For $\ell=1$, in fact, we show that, if $J(\M)$ has linear quotients, then $\M$ must be a uniform matroid.

\begin{Proposition} Let $\M$ be a matroid on $[n]$ of rank $r$ with no loops and assume $J(\M)$ has linear quotients. Then $\M = U_{r,n}$.
\end{Proposition}

\begin{proof} By \cite{Stanley} $R/J(\M)$ is a level algebra. Hence, having linear quotient implies  having a linear resolution, which implies that $J(\M)$ is equigenerated. Since $J(\M)$ is Cohen--Macaulay, and, by \cite{MT2} or \cref{Regularity} $\reg(R/J(\M)) = n - r$, $J(\M)$ is equigenerated in degree $n-r+1$. The generators of $J(\M)$ correspond to cocircuits of $\M$, hence all cocircuits of $\M$ have size $n-r+1$. It quickly follows that $\M = U_{r,n}$.
\end{proof}

We would like to make a stronger statement that if $\symp{J(\M)}$ has linear quotients for some $\ell \geq 2$, then $\M$ must be uniform. It is clear that the proof above fails in this case, since $R/\symp{J(\M)}$ is no longer a level algebra. We leave this question to the reader.
\begin{Question} Let $\M$ be a matroid with no loops. If $\symp{J(\M)}$ has linear quotients for some $\ell \geq 2$, then must $\M$ be uniform?
    
\end{Question}

Despite the above, in the next section we prove that \Cmatroid ideals $J$ are minimally resolved by mapping cones. We extend this result to all symbolic powers $J^{(\ell)}$ in our sequel paper. As one normally does for ideals with linear quotients, we first need to carefully choose a linear ordering of the minimal generators. In the case of linear quotients, there is a way to adjust the ordering so that the resolution by mapping cones is guaranteed to be minimal (see \cite[Lemma~1.5]{HT02} and \cite[Lemma~2.1]{JZ10}). In stark contrast to this case, after choosing an ordering, we still need to prove the minimality of the resolution by iterated mapping cones. To overcome this obstacle, we will establish the connection appearing in \Cref{All-Degrees-Show-Up}. 

First, we introduce some notation.

\begin{Notation}\label{Def-Iterated-Contraction}
Let $\M$ be a matroid with $J=J(\M)$, and let $(v_1,...,v_s) \sub [n]$ be an {\em ordered} subset of vertices. 
To stress the relevance of the order (for future use), we write $\M \con (v_1,...,v_s):=\M \con \{v_1,...,v_s\}$, which equals the iterated contraction $(((\M \con v_1) \con v_2) ... \con v_s)$. Then we use the notation $J \con (v_1,...,v_s) :=J(\M \con (v_1,...,v_s))$.
\end{Notation}

Note that we keep track of the ordering of the vertices for the purpose of defining a total order on $G(J(\M))$. 
We isolate the situation that is most relevant to us.
\begin{Remark}\label{Rmk-filtration-on-G(J)}
Let $\M$ be a matroid of rank $c$ with $J = J(\M)$, and let $(v_1,..,v_c)$ be a basis of $\M$. Set $J_i := J\con(v_1,...,v_{i})$ for $1 \leq i \leq c$. Since $(v_1,\ldots,v_c)$ is independent, then $\h(J\con(v_1,...,v_{i}))=c-i$. We use the filtration
$$J_c=(0)\subseteq J_{c-1} \subseteq J_{c-2} \cdots \subseteq J_i \subseteq \cdots \subseteq J_1 \subseteq J_0:=J,$$ 
to induce the following partition of $G(J)$
$$
G(J)= \bigsqcup_{i=0}^{c-1}(G(J_i) - G(J_{i+1})).
$$
\end{Remark}

\begin{Construction}(Ordering on $G(J(\M))$ by iterated contractions)\label{Con-Ordering}
Let $\M$ be a matroid of rank $c$, let $(v_1,..,v_c)$ be an ordered basis of $\M$, and set $J_0:=J=J(\M)$, $J_i := J\con(v_1,...,v_{i})$. Choose any order on $G(J_i)-G(J_{i+1})$. 
For any $N\in G(J)$, set
\begin{center} 
	$i_N:=$ the index $i$ such that $N\in G(J_i) - G(J_{i+1})$. 
\end{center}
	We set the following total order on $G(J)$:
\begin{center}
 $N_1\stineq N_2$ in $G(J)$\;\;  $\Llra$\;\;  $i_{N_1}>i_{N_2}$, \;\;  or \;\;  $i_{N_1}=i_{N_2}$ and $N_1<N_2$ in $G(J_{i_{N_1}}) - G(J_{(i_{N_1}+1)})$.
\end{center}

In what follows, whenever we write $G(J)=\{N_1,\ldots,N_t\}$ we automatically assume $N_1 \stineq N_2 \stineq \ldots \stineq N_t$.
\end{Construction}

 We will provide an example  describing the ideal of each contraction.
Before the example, for any $N\in G(J)$, \cref{Partition-By-Contraction} provide us with an explicit description of the index $i_N$ of \cref{Con-Ordering}. 

\begin{Remark}\label{Remark-On-Ordering} 
	Let $(v_1,\ldots,v_c)$ be an ordered basis of $\M$. For any $N\in G(J)$,
$$
i_N=\min\{ i\in [c] \,\mid\, x_{v_i} \in \supp{N} \}-1.
 $$   
	In particular, $i_N=0$ $\Llra$ $N \in G(J \div v_1)=G(J) - G(J\con v_1)$, and $i_N = c-1 \Llra N = \min G(J)$.
\end{Remark}

\begin{Example}\label{Example-ordering}
Consider the $C$-matroidal ideal $J = (x_1x_2,x_1x_3,x_1x_4x_5,x_2x_3,x_2x_4x_5,x_3x_4x_5) \subseteq \kk[x_1,\ldots,x_5]$. We choose $(v_1,v_2,v_3) = (1,2,3)$, then a possible ordering on $G(J)$ as in \cref{Con-Ordering} is  $$\underbrace{\underbrace{x_3x_4x_5}_{J_2 =J/(v_1,v_2)} \stineq x_2x_3 \stineq  x_2x_4x_5}_{J_1=J/(v_1)} \stineq \underbrace{x_1x_2 \stineq x_1x_3 \stineq x_1x_4x_5}_{J\div v_1}.$$ 
For $N_1 = x_2x_4x_5$, its index is $i_{N_1} = 1$. For $N_2 = x_3x_4x_5$, its index is $i_{N_2} = 2$, and it is the only monomial with this index.
\end{Example}

As mentioned, one of the important steps to resolve the cover ideal of a matroid by iterated mapping cones is describing the resolution of the colon ideals. Unfortunately, in contrast with the case of linear quotients, our colon ideals are rarely linear ideals. However, \Cref{Corr-Colon-Are-Matroid} shows that, under any of our orderings, the colon ideals are {\em C-matroidal ideals}. So, one could say that we prove that \Cmatroid ideals have ``\Cmatroidal quotients". This opens up a recursive structure that we will crucially utilize.
We first set up some notation.

\begin{Notation}\label{Notation-Colon} If $J=J(\M)$ for some matroid $\M$ and $N \in G(J)$, we write 
	$$
J_{\ineq N} = (N' \in G(J) \, \mid\, N' \ineq N ).	
	$$
Also, 	if $N \neq \min\, G(J)$, we write $$J_{\stineq N} = (N' \in G(J)\,\mid\, N' \stineq N) \qquad \text{ and }\qquad\col{N} := J_{\stineq N} : N.$$
\end{Notation}

The fact that each colon ideal under the ordering of \cref{Con-Ordering} is \Cmatroid is a consequence of yet another characterization of matroids. This characterization is easily obtained by translating the circuit axioms in terms of the colon operation on monomial ideals.

\begin{Proposition}\label{Colon-Circuits-Equivalence} Let $I$ be the cover ideal of a simplicial complex on vertex set $\vs$, and let $\mathcal{C} = \{ \supp{\gamma_N} : N \in G(I) \}$. For $v \in V$ define $I \con {v} = \{ N \in G(I) : x_v \; \nmid \;  N \}$ and $I\div {v} = \{ N \in G(I) : x_v \; |\;  N \}$. Then the following are equivalent:

\begin{enumerate}
    \item $\mathcal{C}$ is a set of circuits of a matroid,
    \item for all $v\in \vs$, and any $N_1\neq N_2 \in I\div{v}$, we have $N_1 : N_2 \in (I\con{v} : N_2)$.
\end{enumerate}
\end{Proposition}

\begin{Corollary}\label{Corr-Colon-Are-Matroid} 
 Let $J$ be the cover ideal of a matroid $\M$ of rank $c$. Order $G(J)$ by iterated contraction along $(v_1,...,v_{c})$ as in \cref{Con-Ordering}, then for any $N \in G(J)$ with $N\neq \min G(J)$, $\col{N} = (J\con(v_1,...,v_{i_N+1})) : N$. In particular, $\col{N}$ is  a \Cmatroidal ideal.
\end{Corollary}

\begin{proof} This result follows by \cref{Colon-Circuits-Equivalence}(2),  \cref{Colon-Matroid-Structure}, and induction.
\end{proof}

We are now ready to combine a number of results and connect back to \cmname matroids. We show that the colon ideals in  the ordering of $G(J(\M))$ in \cref{Con-Ordering}  are the cover ideals of the cofocal matroids of $\M$.

\begin{Theorem}\label{Colons-Are-CocriticalMatroids} 
Let $\M$ be a matroid. Order $G(J(\M))$ as in  \cref{Con-Ordering}. Let $N \in G(J(\M))$. Then, with $i_N$ as in \cref{Con-Ordering}, we have 
$$
   C_N = \begin{cases}
 J(\zcm{\gamma_N}), & \text{ if }i_N=0\\
 &\\
 J(\M'({\gamma_N})^{0}), & \text{ if }i_N>0, \;\text{ where }\M' = \M \con (v_1,...,v_{i_N}).
\end{cases}$$
Note that, for any $N$, we can always find an ordering so that $i_N = 0$.
\end{Theorem}

\begin{proof}
Suppose $i_N = 0$, then by \cref{Corr-Colon-Are-Matroid} the colon ideal $\col{N}= (J \con v_1) : N$. Now by \cref{Colon-Skeleton-Commute} $(J \con v_1) : N = \sfp_2(J):N$. Finally, by \cref{Example-Focal-Matroid}(1) $\sfp_2(J):N = J(\zcm{\gamma_N})$, hence $\col{N}= J(\zcm{\gamma_N})$.

In the case where $i_N > 0$, using the induced ordering on $G(J(\M')) \subseteq G(J(\M))$, we may replace $\M$ by $\M'$, to assume $i_N=0$, then we are done by the above paragraph. \end{proof}

The next result is a key technical ingredient in proving minimality of resolution by iterated mapping cones of \Cmatroid ideals. 

\begin{Proposition}\label{Min-Squarefree-Necc}\label{Min-squarefree-Stronger} Let $J$ be the cover ideal of a matroid $\M$. Let $N \in G(J)$, let $J\setminus N$ be the monomial ideal generated by $G(J)-\{N\}$. 
	If $L\in G((\symp{(J\setminus N) : N)})$, then $LN\in G(J^{(\ell+1)})$.
\end{Proposition}

\begin{proof} By \cref{Colons-Are-CocriticalMatroids} we can choose an ordering on $G(J)$ so that $C_N = (J\setminus N):N = J(\zcm{\gamma_N})$. By \cref{Critical-Matroid-Structure} $\zcm{\gamma_N}$ is a direct summand of $\cm{\gamma_N}$, hence by \cref{Rmk-2.15}, \cref{Ideal-Restriction-Structure}(3), and \cref{Ideal-Restriction-Structure}(2) $G(J(\zcm{\gamma_N}))^{(\ell)} \subseteq G(J(\cm{\gamma_N})^{(\ell)}$. Hence $L \in G(J(\cm{\gamma_N})^{(\ell)}$, so the result follows by \cref{Critical-Matroid-Cover-Ideal} \end{proof}

We will now derive a special case of the above result that is more directly applicable for our proofs in the upcoming sections.

\begin{Corollary}\label{Min-Square-Free} Let $J$ be the cover ideal of a matroid $\M$, and let $v$ be an independent vertex of $\M$. 
Let $N \in G(J \div v)$. For any $L \in \sfp_{\ell}((J\con v) : N))$, we have $LN \in G(\sfp_{\ell+1}(J))$.
\end{Corollary}

\begin{proof} We can choose an ordering so that $N=\max G(J)$ so then
$(J \setminus N) : N = \col{N} =(J\con v) : N$. The result then follows directly from \cref{Min-Squarefree-Necc}.
\end{proof}

\section{Resolution of \Cmatroidal Ideals and Formulas for Homological Invariants}\label{Section-Resolution}

Let $J$ be the cover ideal of a matroid $\M$. We will show that in the ordering of \cref{Con-Ordering}, the iterated mapping cones produces a minimal free resolution of $J$. In the last section, we have described the structure of the colon ideals $\col{N} = J_{\stineq N} : N $. Under the ordering of \cref{Con-Ordering}, each $\col{N}$ are themselves \Cmatroid ideals. 

In general, to prove the minimality of the resolution, one usually needs a very careful and lengthy investigation of the comparison maps of the mapping cones. However, a trivial consequence of the main result of this section shows that, in this setting, these extensive computations are not necessary.  Indeed, we prove that the multigraded Betti numbers are supported in the squarefree parts of the symbolic powers $J$, hence minimality follows automatically. Even more, this fact leads to yet another characterization of matroids. \Cmatroid ideals are precisely the squarefree monomial ideals $J$ whose multigraded Betti numbers are supported in the squarefree parts of the symbolic powers of $J$. This provides an interesting characterization of matroids in terms of graded free resolutions and symbolic powers.

\begin{Notation}\label{Notation-Iterated-Cone-Res} 
For any monomial ideal $J$, we write $\mdeg(J)$ for the set of multidegrees of $G(J)$.\\
For a multigraded free module $F$, we denote the shift of $F$ by a monomial $N$ by $F\shift{N}$. We also write  $\mdeg(F)$ for the set of the multidegrees of its minimal generators. 

For any squarefree monomial ideal $J$ with a total ordering $\ineq$ on $G(J)$, and $N\in G(J)$, 
\begin{itemize}
	\item we let $J_{\ineq N}$ and $\col{N}$ be as in \cref{Notation-Colon};
	\item we write $\colres{N}_{\bullet}$ for a minimal multigraded free resolution of $R/\col{N}$;
	\item  we write $\idealres{N}_{\bullet}$ for the resolution of $J_{\ineq N}$ by iterated mapping cones;
	\item $\col{N}$ is not defined $\Llra$ $N = \min \,G(J)$, in which case we set $\colres{N}_{\bullet}$ to be the minimal multigraded free resolution of $R/(N)$.\\  When we write $\col{N}$, we will implicitly assume that $N > \min \,G(J)$.
\end{itemize}	
\end{Notation}
If $N = \max\, G(J)$, then $\idealres{N}_\bullet$ is a resolution of $J_{\ineq N} = J$ by iterated mapping cones. 

We now state a preliminary result about the shifts that are produced through iterated mapping cones. We omit the proof, as it readily follows from an induction argument.
\begin{Proposition}\label{Iterated-Cone-Shifts}
Let $J$ be any monomial ideal with an ordering $\ineq$ on $G(J)$. With notation as above, we set $N'' = \min G(J)$. Then by iterated mapping cones we have for $h \geq 1$, 
$$\idealres{N}_h = \colres{N''}_h \oplus \bigoplus_{\substack{ N' \in G(I) \\ N'' < N' \ineq N}} \colres{N'}_{h-1}\shift{N'} \quad \text{and}\quad \mdeg(\idealres{N}_h) = \{N''\} \cup \bigcup_{\substack{ N' \in G(I) \\ N'' < N' \ineq N}}\mdeg(\colres{N'}_{h-1}\shift{N'}).$$
\end{Proposition}

We can now state the main result of this section.  

\begin{Theorem}\label{All-Degrees-Show-Up} 
Let $J$ be any \Cmatroid ideal, and $\res_{\bullet}$ any multigraded minimal free resolution of $R/J$. Then $G(\sfp_h(J)) = \mdeg(\res_h) $.

\end{Theorem}

The above result may be stated as follows: For any $h$ and any \Cmatroid ideal $J$, $\sfp_h(J)=\HSIdeal{h-1}{J}$, where $\HSIdeal{h-1}{J}$ is the $(h-1)$-th homological shift ideal introduced by Herzog et al. \cite{HMRZ21}. 
In \cite{HMRZ21} the authors ask for precise conditions ensuring the equality $\HSIdeal{h}{J}= \HSIdeal{1}{\HSIdeal{h-1}{J}}$ holds. They note that there are monomial ideals for which even the inclusion $\HSIdeal{h}{J}\subseteq \HSIdeal{1}{\HSIdeal{h-1}{J}}$ fails. 
They prove that inclusion holds for ideals with linear quotient, but even in this case they have examples where equality fails.

In contrast with the above, it follows immediately from \Cref{All-Degrees-Show-Up} and \Cref{Sqfree-parts}(2) that, for \Cmatroid ideals $J$, the equality $\HSIdeal{h}{J}=\HSIdeal{1}{\HSIdeal{h-1}{J}}$ holds.

A generalization of \Cref{All-Degrees-Show-Up} to the symbolic powers of $J$ will appear in our forthcoming paper \cite{MN2}. A large part of the proof of \Cref{All-Degrees-Show-Up} follows from the following theorem.

\begin{Theorem}\label{Matroid-Resolution} Let $J$ be the cover ideal of a matroid. We order $J$ using \cref{Con-Ordering} with any basis. Then, with \cref{Notation-Iterated-Cone-Res}, we have $$\mdeg(\colres{N}_{h-1}\shift{N}) \subseteq G(\sfp_{h}(J)).$$ 
	It follows by \cref{Iterated-Cone-Shifts}, that for any $N \in G(J)$, $\mdeg(\idealres{N}_h) \subseteq G(\sfp_h(J))$.\\
	In particular, the iterated mapping cones gives a minimal free resolution of $J$. 
\end{Theorem}

\begin{proof}We proceed by induction on $c = \h\, J$. The base case $c = 1$ is trivial, since $J$ is principal.

Let $c > 1$, and let $v$ denote the first vertex in the chosen basis for the ordering in \cref{Con-Ordering}. 
Recall that in the ordering, $G(J\con v) \stineq G(J\div v)$.
First, let $N\in G(J\con v)$. Since $J\con v$ is \Cmatroidal of height $c - 1$, then $\mdeg(\colres{N}_{h-1}\shift{N}) \subseteq G(\sfp_h(J\con v)) \subseteq G(\sfp_h(J))$, where the first inclusion holds by induction, and the second one by \cref{Sym-Powers-Contraction}.

Now, let $N \in J\div v$. By \cref{Corr-Colon-Are-Matroid} $\col{N} = (J\con v) : N$, so it is \Cmatroidal with $\h\, \col{N} = c-1$. Write $\mathbb I^{\col{N}}_\bullet$ for the resolution of $\col{N}$ by iterated mapping cones. We apply induction on $\col{N} = (J\con v) : N$ to see that $\mdeg(\mathbb I^{\col{N}}_{h-1}) \sub \mdeg(\sfp_{h-1}((J\con v) : N))$. By \cref{Iterated-Cone-Shifts}, it remains to prove that for any $L \in G(\sfp_{h-1}((J\con v) : N))$,  one has $LN \in G(\sfp_h(J))$, but this follows from \cref{Min-Square-Free}. \end{proof}

After recording the following observation, we can prove \cref{All-Degrees-Show-Up}. Recall that for any homogeneous ideal $J$ of $R$ and any $i,j\in \NN_0$, the $(i,j)$-{\em graded Betti number} of $R/J$ is $\beta_{i,j}(R/J):=\dim_\kk[{\rm Tor}_i^R(R/J,\kk)]_j$. Similarly, if $J$ is $\ZZ^n$-graded, for any monomial $N\in R$, the $N$-{\em multigraded Betti number} $\beta_{i,N}(R/J):=\dim_\kk[{\rm Tor}_i^R(R/J,\kk)]_N$, i.e. it is the dimension of the $\kk$-vector space spanned by the graded elements of multidegree $N$ in ${\rm Tor}_i^R(R/J, \kk)$.

\begin{Remark}\label{Betti-Formula}
	Let $J$ be \Cmatroid with $\h\, J = c$. Let $N$ be a monomial representing a non-zero multigraded Betti number of $R/J$, then from \cite[Lemma 4.4]{Herzog-Hibi-Zheng-04} $$\beta_{h,N}(R/J) =\beta_{h,N}(R/(J|_{\supp{N}})).$$ Consequently, we have the following formula for the graded Betti numbers of $R/J$ 
	$$\beta_{h,j}(R/J) = \sum_{\substack{N \in G(\sfp_h(J)) \\ \deg(N) = j }}\beta_{h,j}(R/(J|_{\supp{N}})).$$
\end{Remark}

\begin{proof}(\cref{All-Degrees-Show-Up})
``$\supseteq$" holds by \cref{Matroid-Resolution} along with \cref{Iterated-Cone-Shifts}. \\
``$\subseteq$" The inclusion holds for $h = \h\, J$, because $\mdeg(\res_{\h\,J}) \neq \emptyset$ and, by \cite[Cor~3.21]{MN26}, the ideal $\sfp_{\h\, J}(J)$ is principal.  We may assume $h < \h\, J$, and let $N \in G(\sfp_h(J))$. We need to show that the multigraded Betti number $\beta_{h,N}(R/J) \neq 0$. By \cref{Betti-Formula}, $\beta_{h,N}(R/J) = \beta_{h,N}(R/J|_{\supp{N}})$. Since $N\in G(\sfp_{h}(J))$, then $J|_{\supp{N}}$ is a \Cmatroid ideal of height $h$, by \cref{Ideal-Restriction-Structure}(5). Therefore, $\beta_{h,N}(R/J|_{\supp{N}})$ is nonzero by the previous argument, since now $h = \h \, J|_{\supp{N}}$. \end{proof}

We remark that the formula in \cref{Betti-Formula} is recursive. This is because for any $N \in \sfp_h(J)$, $J|_{\supp{N}}$ is again a \Cmatroid ideal. The top Betti number of $J|_{\supp{N}}$ is concentrated in a single degree and can be computed by taking alternating sums of lower Betti numbers. 

Since $R/J$ is minimally resolvable by iterated mapping cones, the Betti numbers of $R/J$ can be obtained from the Betti numbers of the colon ideals which, by \cref{Colons-Are-CocriticalMatroids}, are the ``0-th" summand of $J(\cm{\gamma_N})$.

A different formula for the Betti numbers of $R/I_{\M}$ was already known to Stanley in \cite[Thm~9]{Stanley}.

\begin{Remark} 
The proof of \cref{Matroid-Resolution} shows that at each step the resolution of $J_{\ineq N}$ by iterated mapping cones is minimal. So in particular we obtain minimal resolutions of {\em truncations} of \Cmatroid ideals, i.e. ideals that are equal to $J_{\ineq N}$ for some \Cmatroid ideal $J$, with an ordering as in \cref{Con-Ordering}. 
For instance, consider the ideal $J$ of \Cref{Example-ordering}
$$
J = (x_1x_2, x_1x_3, x_2x_3, x_1x_4, x_2x_4, x_3x_4),
$$ 
where the ordering is done by contracting along $(4,3,2)$. Then, the ideal $I = (x_1x_2, x_1x_3, x_2x_3,$ $ x_1x_4)$ is a truncation of $J$. Hence by the above we know the minimal graded free resolution of $I$.
In particular, for $I$, the multi-degrees that appear in $\res_2$ are $\{x_1x_2x_3, x_1x_3x_4, x_2x_3x_4\}$. These are all minimal generators of $J^{(2)}$. However, only $x_1x_2x_3$ is a minimal generators of $I^{(2)}$.  We point out that $I$ fails to be a \Cmatroid ideal. 
This motivates the following characterization of matroids.

\end{Remark}

\begin{Theorem}\label{Char-Matroid-100}
Let $J$ be a squarefree monomial ideal of $\h(J)\geq 2$. TFAE:
	\begin{enumerate}[(a)]
		\item $J$ is a \Cmatroid ideal;
		\item $\mdeg(\mathbb{F}_\ell) =G(\sfp_\ell(J))$ for some multigraded (not necessarily minimal) free resolution $\FF_\bullet$ of $R/J$, and for all $\ell\geq 1$;
        \item $\mdeg(\mathbb{F}_2) \subseteq G(\sfp_2(J))$ for some multigraded (not necessarily minimal) free resolution $\FF_\bullet$ of $R/J$.
	\end{enumerate} 
\end{Theorem}

\begin{proof} (a) $\Lra$ (b) is proved in \cref{All-Degrees-Show-Up}. (b) $\Lra$ (c) is obvious. (c) $\Lra$ (a) For any monomial ideal $J$ with $G(J)= \{N_1,...,N_r\}$, we define the ideal 
$$\LCM_2(J) = (\{ \LCM(N_i, N_j) : 2 \leq j \leq r, 1 \leq i < j \} ).$$

It is not hard to see that $\mdeg(\mathbb{F}_2) = \mdeg(\LCM_2(J))$, since all of the shifts of $\mathbb{F}_2$ are minimal elements of the form $(N_i : N_j)N_j = \LCM(N_i, N_j) $ for $i < j$.

Then by assumption $\LCM_2(J) \subseteq \sfp_2(J)$. It can be shown by using the circuit axioms that any squarefree ideal $J$ is \Cmatroid if and only if $\LCM_2(J) \subseteq \sfp_2(J)$.
\end{proof}

\begin{Remark} To check the above, one can take $\mathbb{F}_{\bullet}$ to be the Taylor resolution. 
\end{Remark}

In contrast with other characterizations of \Cmatroid ideals, or properties of some \Cmatroid ideals, we give an example illustrating that one may not replace (c) in \Cref{Char-Matroid-100} with ``$\mdeg(\mathbb{F}_\ell) \subseteq G(\sfp_\ell(J))$", or even the stronger condition ``$\mdeg(\mathbb{F}_\ell) = G(\sfp_\ell(J))$", for some $\ell\geq 3$. The example is optimal under several regards -- it has minimal $\ell$, height, number of variables, and initial degree. 
\begin{Example}\label{l>2}
Let $J = (x_1x_2, x_1x_3,x_3x_4,x_4x_5,x_2x_3x_5) \subseteq \kk[x_1,...,x_5]$. It is a Cohen-Macaulay ideal of height $3$ such that $\mdeg(\mathbb{F}_3) =\{x_1x_2x_3x_4x_5\} = G(\sfp_3(J))$, but $J$ is not a \Cmatroid ideal.
\end{Example}

As a final application we give a new, short proof of a known result on regularity and the level property of \Cmatroid ideals \cite[Cor.~of~Thm.~9]{Stanley}, which is an immediate consequence of \cref{Matroid-Resolution}. Recall that, in our context, for any homogeneous ideal $J$, the Castelnuovo--Mumford regularity of $R/J$ 
is ${\rm reg}(R/J)=\max\{j-i\,\mid\,\beta_{i,j}(R/J)\}$.

\begin{Corollary}\label{Regularity} Let $J=J(\M)$ for a matroid $\M$. Then $R/J$ is a level algebra of $\reg(R/J) = |\msupp J | - \h\, J$. In particular, if $\M$ is loopless then $\reg(R/J) = \dim(R/J)$. 
\end{Corollary}

\begin{proof} Let $\res_{\bullet}$ be a resolution of $J$ by iterated mapping cones using the ordering of \cref{Con-Ordering}. By \cref{Matroid-Resolution} this resolution is minimal. Set $c = \h\, J$. Since $J$ is Cohen-Macaulay we know $$\reg(R/J) = \max \{ j : \beta_{c, j}(R/J) \neq 0 \} - \h\, J.$$

By \cref{Matroid-Resolution} the multidegrees of $\res_{c}$ is contained in $G(\sfp_c(J))$. But  $J^{(c)}$ has only one minimal squarefree generator, it is $\prod_{x_i\in \textrm{m-supp}(J)} x_i$. Hence the statements follow. For the in particular, if $\M$ is loopless then every variable is independent, hence $\msupp{J} = \{x_1,...,x_n\}$. So $|\msupp J| - \h\,J = \dim R - \h\,J=\dim(R/J)$.
\end{proof}

\bibliographystyle{amsalpha}
\bibliography{MatroidBib}

\end{document}